\documentclass[a4paper,11pt]{amsart}

\makeatletter
\newcommand*{\rom}[1]{\expandafter\@slowromancap\romannumeral #1@}
\makeatother

\usepackage{amssymb,amscd,amsthm,mathrsfs}
\usepackage[backref=page]{hyperref} 
\usepackage[ansinew]{inputenc}
\usepackage[centertags]{amsmath}
\usepackage{graphicx,psfrag}
\usepackage{xcolor}

\renewcommand*{\backref}[1]{}
\renewcommand*{\backrefalt}[4]{\quad \tiny
  \ifcase #1 (\textbf{NOT CITED.})%
  \or    (Cited on page~#2.)%
  \else   (Cited on pages~#2.)%
  \fi}

\newtheorem*{thm*}{Theorem}

\newtheorem{thm}{Theorem}[section]

\newtheorem{lemma}[thm]{Lemma}

\newtheorem{prop}[thm]{Proposition}
\newtheorem{cor}[thm]{Corollary}

\newcommand{\bi}{\begin{itemize}}
\newcommand{\ei}{\end{itemize}}

\theoremstyle{definition}

\theoremstyle{remark}

\newtheorem*{theorem*}{Theorem}

\newcommand{\T}{\mathbb{T}}
\newcommand{\R}{\mathbb{R}}
\newcommand{\Z}{\mathbb{Z}}
\newcommand{\N}{\mathbb{N}}
\newcommand{\Q}{\mathbb{Q}}

\newcommand{\C}{\mathcal{C}}

\newcommand{\A}{\mathbb{A}}

\author[A. Passeggi]{Alejandro Passeggi}
\address{UdelaR, Facultad de Ciencias.}
\curraddr{Igua 4225 esq. Mataojo. Montevideo, Uruguay.}
\email{alepasseggi@gmail.com}

\author[M. Sambarino]{Mart\'{\i}n Sambarino}
\address{UdelaR, Facultad de Ciencias.}
\curraddr{Igua 4225 esq. Mataojo. Montevideo, Uruguay.}
\email{samba@cmat.edu.uy}

\title[On the Franks-Misiurewicz conjecture]{Deviations in the Franks-Misiurewicz conjecture}

\begin{document}

\maketitle

\begin{abstract}

We show that if there exists a counter example for the rational case of the Franks-Misiurewicz conjecture, then
it must exhibit unbounded deviations in the complementary direction of its rotation set.

\end{abstract}

\section{Introduction}

After the seminal result due to M. Miziurewicz and K. Ziemian \cite{m-z} proving that
the rotation set of a lift $F:\R^2\to\R^2$ of a homeomorphism homotopic to the identity
$f:\T^2:=\R^2/_{\Z^2}\to\T^2$ given by
\small
$$\rho(F)=\left\{\lim_i\frac{F^{n_i}(x_i)-x_i}{n_i}:\ x_i\in\R^2,\ n_i\nearrow +\infty\right\}$$
\normalsize
is a compact convex set, a theory has been developed. Many authors have contributed with different
articles which mostly can be classified under two different focus: (i)
assuming shapes (point, segments, non-empty interior) for the rotation set, derive dynamical properties
(see for instance \cite{franks-reali, franks-reali2, m-z2, llibre-mackay}), (ii) try to find which convex sets are rotation sets
(see for instance \cite{Kwapisznonpol,BoyCaHall}).

\smallskip

Concerning point (ii) there is a long-standing conjecture due to Franks and Misiurewicz \cite{franksmisiu}
which claims the following: if a non trivial interval $I$ is attained as a rotation set then:
\begin{itemize}
\item if $I$ has irrational slope, one end-point is rational,
\item if $I$ has rational slope, it contains a rational point.
\end{itemize}
For the irrational case, A. Avila presented a smooth counter example in 2014 (still not published) which is
minimal. For the second case there have been important progress in the last years.
In \cite{nosotros} it is shown that there can not be a minimal counter example. In fact it is proven that a counter example for this case can not be an
extension of an irrational rotation, and then using the results of Kocsard \cite{Kocsard} and
J\"ager-Tal \cite{jagertal}, one concludes that a minimal example should be an extension of an irrational
rotation, so it can not exist.

\smallskip

In this article, improving \cite{jagertal}, we show that a possible counter example must exhibit
\emph{unbounded deviations} in the complementary direction of the supporting line of the interval $\rho(F)$.
This turns to be quite suggesting as it is shown in several cases that having two different rotation vectors, is
an obstruction for deviations \cite{davalos,forcing}.

\subsection{Precise result}

We call by $\textrm{Homeo}_0(\T^2)$ the family of homotopic to the identity toral homeomorphisms. The rotation
set is defined above. Let $\rho(F)$ be a non-trivial segment contained in a supporting line $\{p+\lambda v\}_{\lambda\in\R}$,
the \emph{perpendicular deviation} of $f$ is given by the (possibly infinite) value
$$\textrm{dev}_{\perp}(f)=\sup_{x\in\R^2} \left\{\textrm{d}(\textrm{pr}_{\perp}(n\cdot\rho(F)),\textrm{pr}_{\perp}(F^n(x)-x))\right\}$$

where $\textrm{pr}_{\perp}:\R^2\to v^{\perp}$ is the projection on a unitary element of $v^{\perp}$, and $\textrm{d}(\cdot,\cdot)$ is euclidean distance in $v^{\perp}$.
We prove the following result.

\begin{thm*}

Assume $f$ is a counter example for the rational case of the Franks-Misiurewicz conjecture. Then it has infinite perpendicular deviation.

\end{thm*}




\subsection{Strategy}

As explained in \cite{nosotros}, in order to obtain the result above, we can just work with vertical rotation sets. So we must show

\begin{thm}\label{T.semiconj}

Assume that for a lift $F$ of $f\in\textrm{Homeo}_0(\T^2)$ we have $\rho(F)=\{\alpha\}\times[\rho^-,\rho^+]$
where $\rho^-<\rho^+$, $\alpha\in\Q^c$. Then $f$ has unbounded horizontal deviation.

\end{thm}

In this last sentence \emph{unbounded horizontal deviation} stands for the value $\textrm{dev}_{\perp}(f)$ being infinite when $\rho(F)$
is vertical.

For proving Theorem \ref{T.semiconj} we suppose a counter example with bounded horizontal deviations is possible, and then by improving \cite{jaertal} we get
that this counter example would be an extension of an irrational rotation. This is absurd since \cite{nosotros}.

\section{Topological results}

We consider the torus given by $\T^2=\R^2/_{\Z^2}$ and $\pi_{\T^2}:\R^2\to\T^2$ the covering map. The annulus
is given by $\A=\R^2/_{\sim}$ where $(u,v)\sim (r,s)$ iff $u=r$ and $v-s\in\Z$. We have the natural
covering maps $\pi:\R^2\to \A$ and $\textrm{p}:\A\to\T^2$. In $\R^2$ we name the projection over the first coordinate
by $\textrm{pr}_1$, and over the second coordinate by $\textrm{pr}_2$.

\smallskip

An \emph{annular continuum} in $\A$ is a continuum so that its complements is
given by exactly two unbounded connected components. A \emph{circloid} in $\A$ is an annular continuum which
is minimal with respect to the inclusion. In this article we call \emph{annular continuum} in $\T^2$ to
$\textrm{p}(A)$ where $A\subset \A$ is an annular continuum and $p|_{A}$ is a homeomorphism. A \emph{circloid}
in $\T^2$ is an annular continuum which is minimal with respect to the inclusion.

\smallskip

Back in $\A$ we can define a partial order in the annular continua. Given an annular continua $A\subset \A$ we
have two unbounded components in its complement. We call $\mathcal{U}^+(A)$ to the one whose lift has a projection under $\textrm{pr}_1$ without upper bound, and by $\mathcal{U}^-(A)$ to the complementary one. For two annular continua $A,B$
in $\A$ we say $A$ \emph{precedes} $B$ iff $B\subset\textrm{cl}[\mathcal{U}^+(A)]$, or equivalently 
$A\subset\textrm{cl}[\mathcal{U}^-(B)]$. We denote this by $A\preceq B$.

\medskip

Consider the following situation which we call by \textbf{(S)} along this article: $\C_1,\C_2$ are circloids in $\A$, and $A\subset \A$ is an annular continuum so that:

\begin{itemize}

\item $\C_1\cap\C_2\neq\emptyset$;

\item $\C_1\preceq A\preceq \C_2$;

\item $\C_1\not\subset A$, $C_2\not\subset A$;

\end{itemize}

The second and third item implies that $\C_1\neq\C_2$.
Moreover, for this setting we have that any connected component of $\C_1\cap \C_2$ must be inessential
and contained in $A$, and the same holds for $\C_1\cap A$ and $A\cap \C_2$. Furthermore, if we consider $\tilde{\C}_1,\tilde{\C}_2$ and $\tilde{A}$ be
lifts of $\C_1,\C_2$ and $A$ respectively, and the family
$$\textrm{CC}=\textrm{c.c.}(\tilde{\C}_1\cap\tilde{\C}_2)\cup \textrm{c.c.}(\tilde{\C}_1\cap \tilde{A})\cup \textrm{c.c.}(\tilde{A}\cap\tilde{C}_2)$$

then there exists $K_0>0$ so that
$$\sup_{X\in\textrm{CC}}\left\{\textrm{diam}(\textrm{pr}_2(X))\right\}< K_0.$$

We introduce now a definition. Given a sub-continuum $Z\subset\tilde{A}$ and $X\in\textrm{c.c.}(\tilde{C}_1\cap\tilde{C}_2)$ we define the
\emph{vertical homotopical intersection number} of $Z$ and $X$ by
$$\nu(X,Z)=\#\{v\in\{0\}\times\Z:\ X+v\subset Z\}.$$

Our goal is to prove the following proposition.

\begin{prop}\label{p.inter}
Let $X\in \textrm{c.c.}(\tilde{\C}_1\cap\tilde{\C}_2)$ and $(Z_n)_{n\in\N}$ be a sequence of sub-continua contained in $\tilde{A}$ with
$\textrm{diam}(Z_n)\to+\infty$. Then, $\nu(X,Z_n)\to+\infty$.

\end{prop}

Before we proceed with the proof we introduce some useful definitions and results. Given two continua $X$ and $Z$ in $\R^2$ we say that $X$ is
$K$-\emph{centered} with respect to $Z$ if $\textrm{pr}_2(Z)\setminus\textrm{pr}_2(X)$ consists in the union of two disjoint intervals
both having length larger than $K$.

Given a continuum $Z$ in $\R^2$ we say that a continuum $Y$ is $K$-\emph{virtually to the right of} $Z$ if
it is $K$-centered w.r.t. $Z$ and there exists a pair of disjoint vertical half-lines $r,s$ so that:

\begin{itemize}

\item $\textrm{pr}_2(r)$ is bounded bellow and $\textrm{pr}_2(s)$ is bounded above;

\item $r$ meets $Z$ only at it starting point $r_0$ which verifies $\textrm{pr}_2(r_0)=\max \textrm{pr}_2(Z)$;

\item $s$ meets $Z$ only at it starting point $s_0$ which verifies $\textrm{pr}_2(s_0)=\min \textrm{pr}_2(Z)$;

\item $Y$ is contained in the closure of the connected component of $\R^2\setminus s\cup Z\cup r$ whose first projection is unbounded
to the right.

\end{itemize}

Note that for any continuum $Z$ of $\R^2$ it always can be considered such a two half-lines $r,s$, and that $s\cup Z\cup r$ defines a unique connected component $\mathcal{R}$ whose first projection is unbounded o the right and a unique connected component $\mathcal{L}$ whose first projection is unbounded to the left. The analogous definition can be considered for $K$-\emph{virtually to the left}. Before presenting a proof for the proposition we state a lemma.

\begin{lemma}\label{l.prevprop}

Assume we have two sequences of planar continua $(Y_n)_{n\in\N}$ and $(L_n)_{n\in\N}$ so that
$Y_n$ is $a_n$-virtually to the left of $L_n$ with $a_n\to_n\infty$, and $\mathcal{L'}=\lim_H\pi(Y_n)$,
$\mathcal{L}=\lim_H\pi(L_n)$ are annular continua. Then $\mathcal{L'}\preceq\mathcal{L}$.

\end{lemma}

\begin{proof}

Suppose that $\mathcal{L'}\not\preceq\mathcal{L}$ for an absurd. Then we can construct a curve $\Gamma:[0,+\infty)\to\A$
whose image is contained in $\mathcal{U}^+(\mathcal{L})$, starting at a point $x_0\in\mathcal{L}'$ and so that $\Gamma(t)\to_{t\to +\infty} +\infty$. Thus we can take a lift $\tilde{\Gamma}$ of $\Gamma$ starting at a lift
$\tilde{x}_0$ of $x_0$, which is contained in $\mathcal{U}^+(\tilde{\mathcal{L}})$. Moreover, we can assume
that $\textrm{pr}_2(\tilde{\Gamma})$ is bounded.

On the other hand, we can consider vertical integer translations $Y_n'\subset \tilde{\mathcal{L}}'$ of the elements $Y_n$ so that $Y'_n\cap B(\tilde{x}_0,\varepsilon_n)\neq\emptyset$ with
$\varepsilon_n\to_n 0$. We claim that this implies the existence of $n_0\in\N$ such that for all $n\geq n_0$ some integer vertical translation $L_n'$ of $L_n$ must meet $B(\tilde{x}_0,\varepsilon_n)$: for this we pick $n_0$ so that
$a_n$ is larger that $\textrm{diam}(\textrm{pr}_2(\tilde{\Gamma}))+2\varepsilon_n$. Thus by taking $L'_n$ for
all $n\geq n_0$ so that $Y'_n$ is $a_n$-virtually to the left of $L'_n$, as $Y'_n$ is contained in the region
to the left of $r\cup L'_n\cup s$ ($r,s$ half lines of the definition of virtually to the left ) with
$(r\cup s)\cap (\tilde{\Gamma}\cup B(\tilde{x}_0,\varepsilon_n))=\emptyset$, we must have
$$(\tilde{\Gamma}\cup B(\tilde{x}_0,\varepsilon_n))\cap L'_n \neq\emptyset,$$
and we are done with the claim.

Hence, we have that $\tilde{x}_0\in \tilde{\mathcal{L}}$, so $x_0\in\mathcal{L}$, which concludes.

\end{proof}

\begin{proof}[Proof of Proposition \ref{p.inter}]

Assume for a contradiction that we have some $X\in \textrm{c.c.}(\tilde{\C}_1\cap\tilde{\C}_2)$ which does not satisfy the proposition. This implies that we can construct a sequence $(W_n)_{n\in\N}$ given by integer vertical translations of some elements of $(Z_n)_{n\in\N}$
so that:

\begin{itemize}
\item[(i)] $X$ is not contained in $W_n$, for any $n\in\N$;

\item[(ii)] $X$ is $n$-centered w.r.t. $W_n$.

\end{itemize}

We will arrive to a contradiction from this situation. As we are in the situation considered above for $\C_1,\C_2$ and $A$, by taking subsequences, we have either the following situation or the symmetric one:
for every $n\in\N$ there exists a point $x_n\in X\setminus W_n$ so that it is $n$-virtually to the right of $W_n$.

Let us assume this situation, for the complementary one the symmetric argument works. In this context we have the
set $s\cup W_n\cup r$ as in the definition of \emph{virtually to the right}, and its \emph{right} component $\mathcal{R}$ with $x_n\in\mathcal{R}$.
We claim the existence of a sequence of continua $L_n\subset\tilde{C}_1$ verifying:

\begin{enumerate}

\item $L_n\cap B(x_n,\frac{1}{n})\neq\emptyset,$

\item $L_n\subset\mathcal{R}$,

\item $\textrm{diam}(\textrm{pr}_2(L_n))>\frac{n}{2}-1$.

\end{enumerate}

For this, we take a reference line $\Gamma:(-\infty,0]\to \mathcal{U}^-(C_1)$ from $-\infty$ to $B(\pi(x_n),\frac{1}{n})$ and lift
it to a line $\tilde{\Gamma}$ in $\R^2$ with image in $U^-=\pi^{-1}(\mathcal{U}^-(C_1))$. We have that
$\tilde{\Gamma}\cap W_n=\emptyset$ (abusing notation by calling the line and its image with the same name), so
$\textrm{diam}(\textrm{pr}_2(\tilde{\Gamma}\cap\mathcal{R}))>\frac{n}{2}$. Moreover $\tilde{\Gamma}\cap\mathcal{R}$
is in a different connected component of $\mathcal{R}\setminus\tilde{C}_1$ than $U^+=\pi^{-1}(\mathcal{U}^+(C_1))$,
in the space $\mathcal{R}$. This implies that some connected component of $\tilde{C}_1\cap\mathcal{R}$ separates
$\Gamma$ from $U^+$ in $\mathcal{R}$. Such connected component, contains a continuum $L_n$ as claimed.



\smallskip

As the $L_n$ constructed are in $\mathcal{R}$, the right region of $s\cup W_n\cup r$, we have the existence
of a continuum $W'_n\subset W_n$ which is $\frac{n}{6}$ virtually to the left of $L_n$, with
$\textrm{diam}(\textrm{pr}_2(W'_n))\to +\infty$: otherwise, we can construct another line $\Gamma'$
joining $-\infty$ to $L_n$ with $\Gamma'\cap (s\cup W_n\cup r)=\emptyset$, which contradicts $L_n\subset\mathcal{R}$.

\smallskip





Taking subsequences
we can assume that $\lim_H \pi(W_n')=\mathcal{L}'$ and $\lim_H\pi(L_n)=\mathcal{L}=\C_1$ both annular continua.
In this situation, Lemma \ref{l.prevprop} implies that $\mathcal{L}'\preceq \C_1$, which under our hypothesis
implies $\C_1\subset A$, which is imposible.

\end{proof}

We call \emph{bunch} to any region $\textrm{cl}[\mathcal{U}^+(C_1)\cap\mathcal{U}^-(C_2)]$
where $C_1,C_2$ are as considered in the situation \textbf{(S)}. An annular continuum $A$ is strongly contained in
a bunch $\mathcal{B}=\textrm{cl}[\mathcal{U}^+(C_1)\cap\mathcal{U}^-(C_2)]$ if it is as in \textbf{(S)}.

\begin{cor}\label{c.noestira}

Assume $\hat{f}\in\textrm{homeo}_0(\A)$ lifts a toral homeomorphisms in $f\in\textrm{homeo}_0(\T^2)$. Further asume that $\mathcal{B}$ is a bunch, $A$ is an annular continuum strongly contained in $\mathcal{B}$, and $Z\subset\R^2$ is a planar continuum with $\pi(Z)\subset A$ so that $f^n(\pi_{\T^2}(Z))\subset \textrm{p}(A)$ for infinitely many positive integers $n$. Then, $Z$ can not contain two points having different rotation vectors for a planar lift $F$ of $f$.

\end{cor}

\begin{proof}

Fix a non-empty connected component $X$ of $\tilde{C}_1\cap\tilde{C}_2$, where  $\textrm{cl}[\mathcal{U}^+(C_1)\cap\mathcal{U}^-(C_2)]$ and $\tilde{C}_1,\tilde{C}_2$ lifts $\mathcal{C}_1,\mathcal{C}_2$. Assume for a contradiction that $Z$ contains points having different rotation vectors.

\smallskip

Then $\textrm{pr}_2(F^n(Z))\to +\infty$, so we have by Proposition \ref{p.inter} that 
the number of integer copies of $X$ contained in $F^n(Z)$ must be unbounded in $n$. 
This is imposible for the lift $F$ of a toral homeomorphism $f$ and a planar continuum $Z$.
 
\end{proof}

In view of this corollary, we now want the following result.

\begin{prop}\label{p.tipocompgen}

Assume an annular continuum $A\subset\A$ is strongly contained in a bunch generated by the cricloids $\mathcal{C}_1$
and $\mathcal{C}_2$. Then, if $z_1,z_2\in A$ are any two points there exists a continuum $Z\subset\tilde{A}$
so that $\pi^{-1}(z_1)\cap Z\neq\emptyset$ and $\pi^{-1}(z_2)\cap Z\neq\emptyset$.

\end{prop}

\begin{proof}

Fix two lifts $z'_1,z'_2\in\tilde{A}$ of $z_1,z_2$ respectively. It is easy to see that we can construct two sequences of continua $(Z^1_n)_{n\in\N}$ and $(Z^2_n)_{n\in\N}$ so that

\begin{itemize}

\item $z'_1\in Z^1_n$ and $z'_2\in Z^2_n$ for all $n\in\N$;

\item $\textrm{diam}(\textrm{pr}_2)(Z^i_n)\to\infty$ for $i=1,2$.

\end{itemize}

If $X$ is any connected component of $\C_1\cap \C_2$ we have due to Proposition \ref{p.inter} that for
some positive integer $n_0$ both numbers $\nu(X,Z^1_n)$ and $\nu(X,Z^2_n)$ are non-zero. As $X$ must be contained
in $A$ we are done.

\end{proof}

\section{Proof of Theorem \ref{T.semiconj}}

In light of the result \cite{nosotros} which forbids the existence of an extension of an irrational
rotation with a rotation set as in the statement of the Theorem \ref{T.semiconj}, in order to conclude
is enough to prove the following intermediate result:

\begin{thm}\label{t.inter}

Assume that for a lift $F$ of $f\in\textrm{Homeo}_0(\T^2)$ we have $\rho(F)=\{\alpha\}\times[\rho^-,\rho^+]$
where $\rho^-<\rho^+$, $\alpha\in\Q^c$, and that $f$ has the horizontal bounded deviation property. Then, some
finite cover of $f$ is an extension of an irrational rotation.

\end{thm}

Thus, by the mentioned result \cite{nosotros}, this can not exists. Our goal now is to prove this last result.

\smallskip

We start by summarizing the constructions in \cite{jager-linear,jagertal}. Fix $f\in\textrm{Homeo}_0(\T^2)$ so that for some lift $F$
we have $\rho(F)=\{\alpha\}\times[\rho^-,\rho^+]$ with $\rho^-<\rho^+$, $\alpha\in\Q^c$, and that $f$ has the horizontal bounded deviation property. In the mentioned article the authors find a family of circloids $\{\mathcal{C}_r\}_{r\in\R}$
of $\A$ having the following properties related to a finite cover of $f$, which we keep calling $f$ (and $\hat{f}:\A\to\A$ to its lift):

\begin{enumerate}

\item $\mathcal{C}_r\preceq\mathcal{C}_s$ whenever $r\leq s$

\item $\mathcal{C}_r\subset B(\pi(\{r\}\times\R),\kappa)$ for some uniform constant $\kappa$;

\item $\hat{f}(\C_r)=\C_{r+\alpha}$;

\item $\textrm{p}(\C_r)$ is a circloid in $\T^2$ for all $r\in\R$.

\item $f^n(\textrm{p}(\C_r))\neq \textrm{p}(\C_r)$ for every $r\in\R$ and every positive integer $n$.

\end{enumerate}

The key result in \cite{jager-linear,jagertal} (see also \cite{JP}) which allows the construction of a semiconjugacy between $f$ and an irrational rotation of angle $\alpha$ is the following.

\begin{thm}

Assume that for some $r_0\in\R$ we have that $f^n(\textrm{p}(\C_{r_0}))\cap f^m(\textrm{p}(\C_{r_0}))=\emptyset$
whenever $n\neq m$, then $f$ is an extension of an irrational rotation.

\end{thm}

Thus, in order to prove Theorem \ref{T.semiconj}, it is enough to see that for some $r\in\R$ the circloid
$\pi(\C_r)$ is free. We assume from now on that for some $r\in\R$ the cricloid $\C_r$ is not free, and construct an absurd throughout this section.

As $\C_r$ is not free, we have that $\C_r\cap\C_s\neq\emptyset$ for some $s\in\R$. We assume $s>r$ (for the symmetric case the same proof works). Thus, due to properties 1 and 5, we have a bunch $\mathcal{B}=\textrm{cl}[\mathcal{U}^+(\C_r)\cap\mathcal{U}^-(\C_s)]$.

Furthermore, due to property 3 we have for some $n_1$ and some $n_2$ that $f^{n_i}(\textrm{p}(\C_r))$ is strongly contained in $\textrm{p}(\mathcal{B})$ for $i=1,2$. This implies that we have for some $r'<s'$ the following
$$\C_r\preceq \C_{r'}\preceq \C_{s'}\preceq \C_s.$$


Define the bunch $\mathcal{B}'$ associated to $r',s'$
Thus, due to property 5, $\mathcal{B}'$ is strongly contained in the bunch $\mathcal{B}$.
Moreover, again due to property 3 and property 2, we have:


\begin{itemize}

\item[(i)] $\textrm{p}(\mathcal{B}'),\dots, \textrm{p}(\hat{f}^{j_0}(\mathcal{B}'))$ covers $\T^2$, for some $j_0\in\N$;

\item[(ii)] $f^n(\textrm{p}(\mathcal{B}'))$ is strongly contained in $\textrm{p}(\mathcal{B})$ for every $n$ contained
in a syndetic set $\mathcal{I}\subset\N$.

\end{itemize}

Property (i) implies that we can find in any lift $\tilde{\mathcal{B}}'$ two points $b^-,b^+$ having rotation vectors
$(\alpha,\rho^+)$ and $(\alpha,\rho^-)$ respectively. Furthermore, as $\mathcal{B}''$ si strongly contained in $\mathcal{B}$, Proposition \ref{p.tipocompgen} allows us to find a continuum $Z\subset\tilde{\mathcal{B}}''$
containing points in the equivalence class of $b^-$ and of $b^+$. But this situation together with point (ii) implies a contradiction of Corollary \ref{c.noestira}.

Therefore, we obtain the desired absurd, which proves \ref{t.inter} and so \ref{T.semiconj}.

\bibliographystyle{koro}
\bibliography{bib}

\end{document}